\documentclass{amsart}
\usepackage{amsfonts,amssymb,amsmath,amsthm}
\usepackage{url}
\usepackage{enumerate}

\usepackage{graphicx}
\usepackage{tikz}
   \usetikzlibrary{calc}
\usepackage{mathptmx}
\usepackage{helvet}
\usepackage{courier}
\usepackage{txfonts}
\usepackage{tikz} 
\usetikzlibrary{arrows}
\usepackage{type1cm} 
\input xy
\xyoption{all}

\newtheorem{theorem}{Theorem}[section]
\newtheorem{lemma}[theorem]{Lemma}

\newtheorem{proposition}[theorem]{Proposition}
\newtheorem{corollary}[theorem]{Corollary}
\theoremstyle{definition}
\newtheorem{definition}[theorem]{Definition}
\newtheorem{remark}[theorem]{Remark}
\newtheorem{example}[theorem]{Example}
\numberwithin{equation}{section}

\newcommand{\B}{\mathrm{B}}
\newcommand{\cB}{\overline{\mathrm{B}}}
\newcommand{\sRips}{\mathrm{sRips}}
\newcommand{\Rips}{\mathrm{Rips}}

\newcommand{\NN}{\mathbb{N}}
\newcommand{\RR}{\mathbb{R}}
\newcommand{\FF}{\mathbb{F}}

\newcommand{\M}{\mathcal{M}}
\newcommand{\locmin}{\mathrm{LocMin}}
\newcommand{\e}{\varepsilon}
\newcommand{\f}{\varphi}

\newcommand{\diam}{\operatorname{Diam}\nolimits}
\newcommand{\cRips}{\operatorname{\overline{Rips}}\nolimits}
\newcommand{\rank}{\operatorname{rank}\nolimits}
\newcommand{\mgs}{\operatorname{mgs}\nolimits}

\title[Critical edges in Rips complexes and persistence]
{Critical edges in Rips complexes and persistence}

\author{Peter Goričan}
\address{{Faculty of Mathematics and Physics, University of Ljubljana, Jadranska 21, SI-1000 Ljubljana, Slovenia} and 
{Institute IMFM, Jadranska 19, SI-1000 Ljubljana, Slovenia}}
\email{peter.gorican@imfm.si}

\author{\v Ziga Virk}
\address{{Faculty of Computer and Information Science, University of Ljubljana, Ve\v cna pot 113, SI-1000 Ljubljana, Slovenia} and
{Institute IMFM, Jadranska 19, SI-1000 Ljubljana, Slovenia}}
\email{ziga.virk@fri.uni-lj.si}
\thanks{Research was  supported by Slovenian Research Agency grants No. N1-0114, J1-4001, J1-4031, and P1-0292.}

\subjclass[2020]{}
\begin{document}

\begin{abstract}
We consider persistent homology obtained by applying homology to the open Rips filtration of a compact metric space $(X,d)$. We show that each decrease in zero-dimensional persistence and each increase in one-dimensional persistence is induced by local minima of the distance function $d$. When $d$ attains local minimum at only finitely many pairs of points, we prove that each above mentioned change in persistence is induced by a specific critical edge in Rips complexes, which represents a local minimum of $d$. We use this fact to develop a theory (including interpretation) of critical edges of persistence. The obtained results include upper bounds for the rank of one-dimensional persistence and a corresponding reconstruction result. Of potential computational interest is a simple geometric criterion recognizing local minima of $d$ that induce a change in persistence. We conclude with a proof that each locally isolated minimum of $d$ can be detected through persistent homology with selective Rips complexes. The results of this paper offer the first interpretation of critical scales of persistent homology (obtained via Rips complexes) for general compact metric spaces.
\end{abstract}

\maketitle
\begin{center}
\today
\end{center}

\renewcommand{\thefootnote}{\fnsymbol{footnote}} 
\renewcommand{\thefootnote}{\arabic{footnote}}

\section{Introduction}


Given a metric space $X$ and a scale $r \geq 0$, there are various constructions that assign a simplicial complex to $X$ at scale a $r \geq 0$: Rips complex, \v Cech complex, alpha complex, etc. A collection of any of the mentioned complexes for all $r \geq 0$ yields a filtration, an increasing sequence of simplicial complexes representing $X$ at all scales. In the past hundred years, filtrations were used to study metric spaces from infinitesimal (shape theory) and asymptotic (coarse geometry) point of view. At the turn of the century, filtrations emerged as one of the foundational concepts of persistent homology, which is obtained by applying a homology to a Rips filtration and is a stable descriptor of metric spaces. This point of view is being used to great effect in topological data analysis. Its computational convenience was established with the first persistence algorithm \cite{ELZ} for finite filtrations. A simplified topological idea used in the algorithm is the following: adding an edge to a simplicial complex either decreases $H_0$ or increases $H_1$. Elaborating on this idea we can see that each decrease $H_0$ or increase $H_1$ can be assigned to a specific edge. This fact which is crucial to extract corresponding homology representatives with the aim to identify those geometric features in our space, that generate parts of persistent homology. 

In this paper we study analogous results for persistent homology obtained by applying Rips filtration to a compact metric space. In this setting it is not apparent what the critical edges corresponding to changes in persistence are and whether they exist (in fact, in general they don't). Besides being of theoretical interest of its own, the stability result of persistent homology imply that our results describe the limit of persistent homologies obtained from ever finer finite samples of $X$. As such our results interpret and provide additional structure to practical computations of persistent homology. 

The \textbf{main results} of this paper are:
\begin{itemize}
 \item Theorems \ref{th:1} and \ref{ThmDesc}: Each scale $c$ where $H_0$ decreases or $H_1$ decreases is in the closure of local minima of the distance function $d$.
 \item Theorem \ref{th:1} and Proposition \ref{PropAdd}: When the cardinality of pairs at which $d$ attains a local minimum is finite, each mentioned change in persistent homology corresponds to specific pairs at which $d$ attains a local minimum.
\end{itemize}

Expanding on this context we use our main results to develop a theory of critical edges of persistent homology. Our \textbf{secondary results} include:
\begin{itemize}
 \item Theorem \ref{ThmReconstr}: Reconstruction result for first homology and fundamental group of compact metric spaces.
 \item Corollary \ref{CorAdd1} and Theorem \ref{ThmFG}: Bounds on the change of $H_1$ and $H_0$ at critical scales of persistent homology, and bounds on the ranks of the same groups. These include the first results on finiteness of the rank of $H_1$ for Rips complexes of non-finite spaces at all scales. 
 \item Theorem \ref{ThmConverse}: A simple combinatorial criterion to determine whether the unique pair of points at which $d$ attains an isolated local minimum causes change in persistent homology. 
 \item Theorem \ref{ThmSRips1}: Detecting each strict local minimum of $d$ through persistent homology via selective Rips complexes.
\end{itemize}
These results are complemented by examples demonstrating the necessity of our assumptions. While some of the results of this paper might seem as expected analogues of finite filtrations, the mentioned examples demonstrate that the technicalities of the analogy are far from straightforward. Almost all our results fail to hold in case $X$ is not compact or in case we use closed Rips filtrations instead of open ones. 

While our primary focus is persistent homology, our results are stated and proved so that they also apply to persistent fundamental group. We use term ``persistence'' to encompass persistent homology and persistent fundamental group.
Figure \ref{Fig3} demonstrates our results on a simple example. 

\begin{figure}
    \centering
    \includegraphics{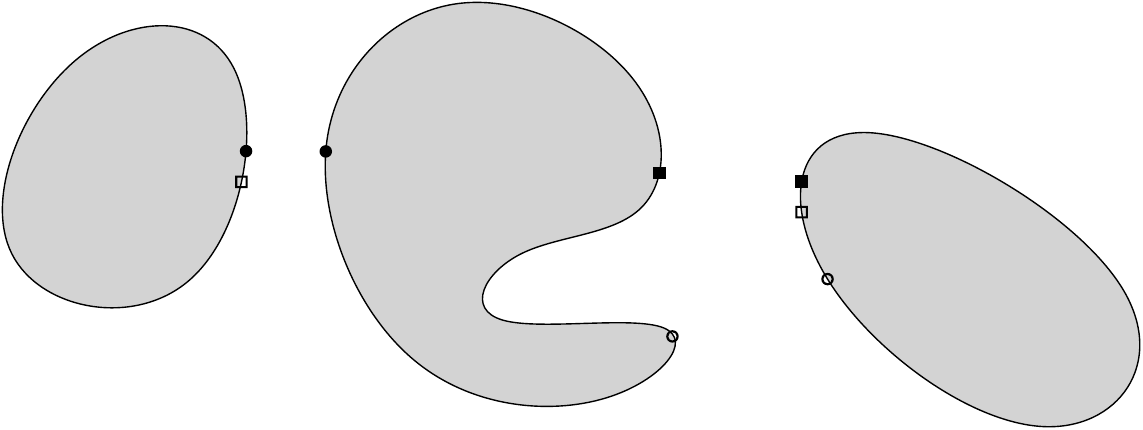}
    \caption{Space $X$ consisting of three components. By Theorems \ref{th:1} and \ref{ThmDesc}, a decrease in $H_0(\Rips(X,r))$ or an increase in $H_1(\Rips(X,r))$ can only be achieved at scales $r$ corresponding to local minima of $d$. Distance function $d$ attains four non-zero local minima at pairs of points of identical appearance. Let us denote the distinct local minima (i.e., the distances between the corresponding points) by $min_\bullet, min_\circ min_\blacksquare,$ and $min_\square$. At $min_\bullet$ and $min_\blacksquare$ two components are merged and thus by Theorems \ref{th:1} and \ref{ThmFG} the only change is decrease in $H_0$. At $min_\circ$ the rank of $H_1$ increases by $1$ by Theorems \ref{ThmDesc}, Theorem \ref{ThmFG}, and \ref{ThmConverse}. At $min_\square$ there is no change to $H_0$ or $H_1$ by Theorem \ref{ThmConverse}, although by using appropriate selective Rips complexes we can see an increase in the rank of $H_1$ by Theorem \ref{ThmSRips2}.}
    \label{Fig3}
\end{figure}
 
\textbf{Related work}.
To the best of our knowledge, specific identifications of simplices inducing a change in persistent homology (via Rips complexes) of non-finite compact spaces has only been carried out in \cite{ZV} and applied in \cite{ZV1}. There it was shown that simplices terminating one-dimensional persistence on compact geodesic spaces are equilateral triangles. On a similar note, critical simplices of \v Cech filtration of a finite collection of points in an Euclidean space were studied in \cite{BauerEdels}. While one-dimensional persistence of nerves was implicitly studied in \cite{Brazas1}, the approach did not utilize specific simplices but rather used Spanier groups. 

There is a growing volume of work studying Rips complexes of simple spaces \cite{AA, Ad5, Feng, Saleh, Shukla} or interpreting parts of persistence diagrams with properties of the underlying space \cite{AC, Lim, ZV, ZV1, ZVCounterex, ZV2}. Reconstruction results using Rips complexes are typically concerned with reconstructing the homotopy type of the underlying space \cite{Haus, ZV3}. To the best of our knowledge, Theorem \ref{ThmReconstr} is the first reconstruction theorem aimed at reconstructing persistent homology at only certain dimension under more general assumptions. 

Our results on the first homology of a Rips complex being finitely generated under certain assumptions complement results in  \cite{Cha2}, which focused on \v Cech complexes. Furthermore,  \cite{Cha2} provided an example of a space, whose closed Rips complex has first homology group of infinite rank. The question of such an example for open Rips complexes was left open. We provide it in Example \ref{Ex3}.

Selective Rips complexes of the last section were first introduced in  \cite{ZV4} for the same reason they are utilized here: to detect more geometric features of a metric space. In \cite{ZV4} the features in question were certain simple closed geodesics, in this paper they are local minima of $d$.

\textbf{Structure of the paper}. 
In Section \ref{SectPrelim} we provide preliminaries. In Section \ref{Sect0Dim} we provide a complete description of zero-dimensional persistence. Section \ref{Sect1Dim} is the most extensive and contains a thorough analysis of emerging homology of one-dimensional persistence, along with a majority of the secondary results. Selective Rips complexes and the way to use them to detect more local minima than with Rips complexes are provided in Section \ref{SectsRips}. Section \ref{SectEx} provides several (counter)examples that demonstrate necessity of our conditions and justify our choices, such as $X$ being compact and use of open Rips complexes instead of closed ones.

\section{Preliminaries}
\label{SectPrelim}

Let $X=(X,d)$ be a metric space with $d\colon X\times X \rightarrow [0,\infty)$. For each $r>0$ and $x\in X$, let $\B(x,r)= \{y\in X \mid d(x,y) < r \}$ be the open ball, and let $\cB(x,r)= \{y\in X \mid d(x,y) \leq r \}$ be the closed ball.
Let $\locmin (d)$ denote the collection of all local minima of the distance function $d$. Note that $0\in \locmin (d)$ if $X \neq \emptyset$. Scale $c$ is an \textbf{isolated local minimum} of the distance function $d$ if it is the only local minimum of $d$ on some open interval containing $c$.

Given a scale $r \geq 0$ the (open) \textbf{Rips complex} $\Rips(X,r)$ is an abstract simplicial complex defined by the following rule: A subset $\sigma \subset X$ is a simplex iff $\diam(\sigma) < r$. In particular, $\Rips(X,0)$ is the empty set and for each $r>0$ the vertex set of $\Rips(X,r)$ is $X$. 

Given a scale $r \geq 0$ the \textbf{closed Rips complex} $\cRips(X,r)$ is an abstract simplicial complex defined by the following rule: A subset $\sigma \subset X$ is a simplex iff $\diam(\sigma) \leq r$.

In what follows we will define persistent homology based on open Rips complexes. Analogue construction can be made with closed Rips complexes, or any other construction yielding a well defined filtration (such as open and closed \v Cech complexes, selective Rips complexes defined in a later section, etc.). On the other hand, we can (and will) perform the same construction with fundamental groups. Our general term ``persistence'' will refer to persistent homology or persistent fundamental group  applied to any filtration, although we will focus on Rips filtrations and its variations.

The \textbf{Rips filtration} of $X$ is the collection of Rips complexes $\{\Rips(X,r)\}_{r \geq 0}$ along with natural inclusions $i_{s,t} \colon \Rips(X,s) \hookrightarrow \Rips(X,t)$ for all $s \leq t$. We will be using $i$ to denote maps of the form $i_{s,t}$ without specifying indices. 

Applying homology $H_q$ with coefficients in $G$ to a filtration we obtain \textbf{persistent homology} consisting of homology groups $\{H_q(\Rips(X,r);G)\}_{r \geq 0}$ and the induced homomorphisms $\{(i_{s,t})_*\}_{s \leq t}.$ In the literature persistent homology is sometimes sometimes denoted as the collection of the ranks of maps $(i_{s,t})_*$. Throughout this paper, persistent homology (or persistent fundamental group) will be the object obtained by applying homology (or fundamental group) to a filtration. The Abelian group $G$ forming coefficients will be considered fixed and omitted from notation. 

Fixing dimension $q\in \{0,1,\ldots\}$ we say a scale $a \geq 0$ is a \textbf{regular value} for $H_q$ if there exists $\e>0$ such that for all $s,t\in (a - \e, a + \e) \cap [0,\infty)$ satisfying $ s \leq t$, the map $(i_{s,t})_*$ is an isomorphism (see \cite{Govc} and \cite{Bubenik} previous appearances of this concept). A scale $a \geq 0$ is a \textbf{critical value} of $H_q$ if it is not a regular value. It will be beneficial if we further distinguish critical values. A scale $a \geq 0$ is an \textbf{emergent-regular value} of $H_q$ if there exists $\e>0$ such that for all $s,t\in (a - \e, a + \e) \cap [0,\infty)$ satisfying $ s \leq t$, the map $(i_{s,t})_*$ is surjective, and is \textbf{emergent-critical} if it is not emergent-regular. Similarly, scale $a \geq 0$ is an \textbf{terminally-regular value} of $H_q$ if there exists $\e>0$ such that for all $s,t\in (a - \e, a + \e) \cap [0,\infty)$ satisfying $ s \leq t$, the map $(i_{s,t})_*$ is injective, and is \textbf{terminally-critical} if it is not terminally-regular. In a similar way we define regular and critical values of persistent fundamental group.

We will refer to the collection of critical values in any of the mentioned cases as the \textbf{spectrum}. For example, the emergent $H_1$ spectrum is the collection of emergent-critical values of $H_1$. Note that by definition each spectrum is a closed subset of $[0,\infty)$. 

 Given a  field $\FF$ and an interval $J \subset [0,\infty) $, the \textbf{interval module} $\FF_{J}$ is a collection of $\FF$-vector spaces $\{V_r\}_{r \in [0,\infty)}$ with 
\begin{itemize}
 \item $V_r = \FF$ for $r\in J$;
 \item $V_r=0$ for $r \notin J$,
\end{itemize}
 and commuting linear bonding maps $V_s \to V_{t}$ which are identities whenever possible (i.e., for $s,t\in J$) and zero elsewhere. When persistent homology of a Rips filtration of a compact metric space $X$ is computed with coefficients in a field $\FF$, it decomposes (uniquely up to permutation of the summands) as a direct sum of interval modules (see q-tameness condition in Proposition 5.1 of \cite{Cha2}, the property of being radical in \cite{ChaObs}, and the main result in \cite{ChaObs} along with its corollaries for details). The  intervals determining the said collection of interval modules are called \textbf{bars}. They form a multiset called \textbf{barcode} of the persistence homology. For each bar, its endpoints form a pair of numbers from $(0,\infty)\cup \{\infty\}$, with the left endpoint being smaller than $\infty$. These pairs form a multiset called a \textbf{persistence diagram}.

\section{Zero-dimensional persistence}
\label{Sect0Dim}

In this section we will analyze critical scales of persistent zero-dimensional homology of a compact metric space $X$ obtained through the Rips filtration. It is apparent that the only emergent-critical value of $H_0$ is zero as $\Rips(X,0)=\emptyset$ while the vertex set of $\Rips(X,r)$ for each $r>0$ is $X$. We thus turn our attention to terminally-critical values.

\begin{lemma}
\label{le:1}
[Finiteness property]
Let $X$ be a compact metric space. Then for every $r>0$, $H_0(\mathrm{Rips}(X,r))$ is finally generated. The rank of $H_0(\mathrm{Rips}(X,r))$ is bounded from above by the number of $r/2$ balls required to cover $X$.
\end{lemma}

\begin{proof}
Let be $r > 0$.  $H_0((\Rips(X,r))$ is  is generated by elements of the form $[v_n]$, where $v_n \in X$.
Since $X$ is compact, we can cover it with finitely many open balls of radius $\frac{r}{2}$, i.e., $\bigcup_{j=1}^k \B(v_j,r/2)=X$. 
For each $x\in \B(v_n,r/2)$ we have $[x]=[v_n]\in H_0((\Rips(X,r))$ hence $[v_1], \ldots [v_k]$ is a finite generating set.
\end{proof}

Let $\rho>0$. A \textbf{finite $\rho$-sequence} between points $x,y\in X$ is a sequence $x=x_1, x_2, \ldots, x_p=y$ of points in $X$ such that $d(x_j, x_{j+1})\leq \rho$.
A \textbf{finite strictly $\rho$-sequence} between points $x,y\in X$ is a sequence $x=x_1, x_2, \ldots, x_p=y$ of points in $X$ such that $d(x_j, x_{j+1})<  \rho$.

\begin{definition}
Let $X$ be a compact metric space and $r>0$. 
We define an equivalence relation $\sim_r$ on $X$ by $x\sim_r y \Leftrightarrow [x]=[y] \in H_0(\mathrm{Rips}(X,r))$. Equivalently, $x\sim_r y$ if there is a finite strictly $r$-sequence between them.
\end{definition}

\begin{lemma}
\label{le:1a}
This equivalence classes of $\sim_r$ are open and closed subsets of $X$.
\end{lemma}
\begin{proof}
 If $x \in X$ and $y\in \mathrm{B}(x,\frac{r}{2})$, then $[x] =[y]$ and thus the equivalence class $[x]$ is open in $X$. According to the Lemma \ref{le:1} there are only finitely many equivalence classes, so they are also closed.
\end{proof}

\begin{theorem}{}
\label{th:1}
[Geometry of terminally-critical scales]
Let $X$ be a compact metric space. Then:
\begin{enumerate}
    \item The only potential accumulation point of the $H_0$ critical values is $0$.
    \item Each terminally-critical value of persistent $H_0$ (i.e., $\{H_0(\Rips(X,r))\}_{r \geq 0}$) is a local minimum of the distance function $d$.
    \item Assume $a_1 < a_2$ are consecutive $H_0$ terminally-critical values and $r\in(a_1,a_2]$, or $a_1$ is the largest $H_0$ critical value and $r>a_1$. Then points $x,y\in X$ satisfy $x \sim_r y$ iff there is a finite $a_1$-sequence between them.  
\end{enumerate}
\end{theorem}

\begin{proof}
Let $r_1>0$ and assume $A_1,A_2,\ldots A_n \subseteq X$ are the finitely many (by Lemma \ref{le:1}) equivalence classes of $\sim_{r_1}$. It is apparent that there at most $n-1$ critical scales of $H_0$ larger than $r_1$, with each of them being a scale at which at least two equivalence classes merge. This implies (1) as $r_1>0$ can was chosen arbitrarily. 
Define 
$$
d(A_j,A_k)=\min\{d(x_j,x_k) | x_j \in A_j, x_k \in A_k \}, \ j,k=1,2,\ldots n.
$$ 
Let be $c_1=\min_{j\neq k}\{d(A_j,A_k)\}$. By Lemma \ref{le:1a} we may assume $c_1=d(A_1,A_2)=d(x_1,x_2)$ for some $x_1\in A_1, x_2\in A_2$. For each $s > c_1 ,\ [A_1]=[A_2] \in  H_0(\mathrm{Rips}(X,s))$.  Furthermore, for each $r \leq c_1, \ [A_j] \neq [A_k] \in H_0((\mathrm{Rips}(X,r))$ for all $j \neq k$ as every pair of points in $X$ at distance less than $r$ is contained in the single class $A_j$ by the definition of $c_1$. Therefore $c_1$ is the first critical values of $H_0$ larger than $r_1$. We claim $c_1$ is a local minimum of the distance function $d$. 

Assume that $d$ does not attain a local minimum at $d(x_1,y_2)$. Then there exist $x'_1 \in A_1$ and $x'_2 \in A_2$ such that $d(x'_1,x'_2)< d(x_1,x_2)$ (because $A_1$ and $A_2$ are open.) This is contradiction because $d(x_1,x _2)=d(A_1,A_2)$. Hence $c_1$ is a local minimum of $d$ attained at a closest pair of points $(x_1, x_2)$.
As there are only finitely many critical scales of $H_0$ larger than $r_1$ we may proceed by induction: set $r_2=c_1$ and repeat the argument for $r_2$ instead of $r_1$, etc. We thus obtained (2). 

In order to prove (3) we may assume $a_1=c_1$. From the argument above observe that $a_1$ is the minimum of $\min_{j\neq k}\{d(A_j,A_k)\}$ while $a_2$ - if finite - is the second smallest number of $\min_{j\neq k}\{d(A_j,A_k)\}$. Since $x \sim_r y$ there is a finite strictly $r$-sequence between them. Each consecutive pair of points $(x'_j, x'_{j+1})$ from this sequence can be replaced by a finite $c_1$-sequence as follows:
\begin{itemize}
    \item If both $x'_j$ and $x'_{j+1}$ are from the same $A_k$, then they may be connected by a finite strictly $c_1$-sequence.
    \item Assume $x'_j \in A_{k_1}$ and $x'_{j+1}\in A_{k_2}$ for $k_1 \neq k_2$. Then $d(A_{k_1},A_{k_2}) < r$ and thus  $d(A_{k_1},A_{k_2}) = c_1$. Without loss of generality we may assume $k_1=1$ and $k_2=2$ given the setting at the beginning of the proof.  Then we may connect:
        \begin{itemize}
            \item $x'_j$ to $x_1$ by a finite strictly $c_1$-sequence.
            \item $x_1$ to $x_2$ by the obvious finite $c_1$-sequence $x_1,x_2$.
            \item $x'_{j+1}$ to $x_2$ by a finite strictly $c_1$-sequence.
        \end{itemize}
\end{itemize}
As a result we obtain a finite $c_1$-sequence from $x$ to $y$ thus (3) holds.
\end{proof}

\begin{remark}
Let us summarize some of the the obtained results:
\begin{enumerate}
    \item The only emergent-critical value of $H_0$ is zero.
    \item Each terminally-critical value of $H_0$ is a local minimum of $d$, i.e., the $H_0$ spectrum is contained in $\locmin(d)$.
    \item The collection of terminally critical values of $H_0$ is either finite or forms a sequence converging towards zero.  
    \item Given a terminally-critical value $c$ of $H_0$ choose $\e>0$ such that no other critical value lies in $(c-2\e, c+2\e)$. Then for each non-trivial $[\alpha] \in \ker i_{c-\e, c+\e} \subset H_0(\Rips(X,c-\e))$, the $0$-chain $\alpha$ is non-trivial in $H_0(\Rips(X, c))$. In particular, $[\alpha]$ (which may, for example, represent the formal difference of two components about to merge at $c$) becomes trivial beyond $c$, but not yet at $c$.
    \item When persistent homology is computed with coefficients from a field, statement (4) implies that all bars of $H_0$ are open at the left endpoint $0$ and closed at the right endpoint.
\end{enumerate}
\end{remark}

\begin{remark}
Statement (1) above holds for any space $X$ for $H_0$, but not for higher-dimensional persistent homology. Statement (2) does not hold if $X$ is not compact: consider the union of the graphs of functions $1 + 1/x$ and $-1-1/x$ for $x>0$, which has a critical value $2$ but the distance function lacks a positive local minimum. Statement (3) follows from Lemma 
\ref{le:1} and also holds for totally bounded spaces, but obviously not in general. Statement (4) holds for any space $X$.

If we are using the closed Rips filtration (upon a compact metric space $X$) instead of the open one, statements (1)-(3) remain the same, while (4) changes: $c$ is the minimal scale at which chain $[\alpha]$ is trivial. Statement (5) changes to: all bars are intervals closed at the left endpoint $0$ and open at the right endpoint.
\end{remark}

\section{One-dimensional persistence}
\label{Sect1Dim}

In this section we will analyze emergent spectrum of persistent $H_1$ and persistent fundamental group of a compact metric space $X$ obtained through the Rips filtration. We first prove that the mentioned spectra are contained in the closure $\overline{\locmin(d)}$ of the local minima of $d$.

\subsection{Geometry of spectra}

\begin{definition}
\label{DefDescent}
Fix $r>0$. Let $(X,d)$ be a compact metric and $x,y\in X$ with $d(x,y)<r$. Choose $\nu< r - d(x,y)$. We say that $(x,y)$ $\nu$-\textbf{descends} (or simply descends) to $(x',y')$ if there are finite $\nu$-sequences $x=x_1, x_2, \ldots, x_p=x'$ and $y=y_1, y_2, \ldots, y_p=y'$ such that for each $j$ we have $d(x_j,y_j) < d(x,y)$, see Figure \ref{Fig1}.
\end{definition}

\begin{figure}
    \centering
    \includegraphics{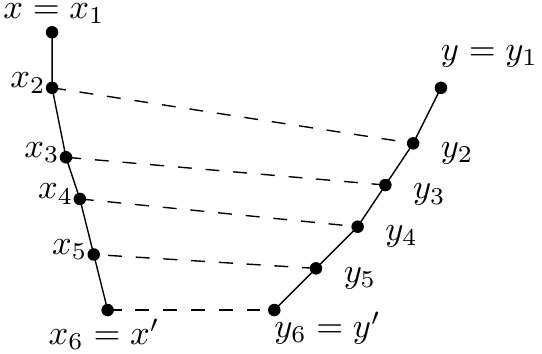}
    \caption{A sketch of a $\nu$-descent. The distances between consecutive points (solid lines) are smaller than $\nu$. The dashed distances are smaller than $d(x,y)$.  }
    \label{Fig1}
\end{figure}

The concept of descending will allow us to replace an edge in a simplicial loop (or a homology cycle) with a sequence of shorter edges without changing the homotopy (homology) type. Descending Lemma below states that we can always descend so that the lengths of the obtained edges are at most the first local minimum of $d$ smaller than $r$ if such a local minimum exists.  To that end we introduce the following notation.

Let $\rho>0$. A \textbf{finite $\rho$-cycle} is a finite $\rho$-sequence $x=x_1, x_2, \ldots, x_p=x$ of points in $X$ (in particular, $d(x_j, x_{j+1})\leq \rho$). A finite $\rho$-cycle will often be identified with a cycle in $\cRips(X,\rho)$ defined as $\sum_{j=1}^{p-1} \langle x_j, x_{j+1} \rangle$.
A \textbf{finite strictly $\rho$-cycle} is a finite strictly $\rho$-sequence $x=x_1, x_2, \ldots, x_p=x$ of points in $X$ (in particular, $d(x_j, x_{j+1}) < \rho$). A finite $\rho$-cycle will often be identified with a cycle in $\Rips(X,\rho)$ defined as $\sum_{j=1}^{p-1} \langle x_j, x_{j+1} \rangle$.

Let $\bullet$ be a basepoint in $X$ and all its Rips complexes unless explicitly stated otherwise.

\begin{lemma}  
\label{LemmaDesc} 
[Descending Lemma]
Assume $X$ is a compact metric space and let $r > 0$. 
\begin{enumerate}
 	\item If $c > 0$ is the only local minimum of the distance function $d$ on the interval $[c,r)$, then:
	\begin{enumerate}
 		\item Each pair of points $x,y\in X$ with $d(x,y)<r$ descends 	to a pair of points $(x',y')$ at distance at most $c$.
  		\item For each $1$-cycle $\alpha$ in $\Rips(X,r)$ there exists a finite $c$-cycle $\alpha'$ in $X$ such that $[\alpha']=[\alpha]\in H_1(\Rips(X,r)).$
		\item Each based simplicial loop $\alpha$ in $\Rips(X,r)$ there exists a based simplicial loop $\alpha'$ in $\cRips(X,c)$ such that $\alpha'\simeq \alpha \ \mathrm{ rel } \bullet $ in  $\Rips(X,r).$
	\end{enumerate}
\item If for some $a \geq 0$ the distance $d$ has no local minima on the interval $(a,r)$, then for each $\e > 0$:
\begin{enumerate}
 	\item Each pair of points $x,y\in X$ with $d(x,y)<r$ descends 	to a pair of points $(x',y')$ at distance at most $a + \e$. 
	\item For each $1$-cycle $\alpha$ in $\Rips(X,r)$ there exists a finite $(a + \e)$-cycle $\alpha'$ such that $[\alpha']=[\alpha]\in H_1(\Rips(X,r)).$
		\item Each based simplicial loop $\alpha$ in $\Rips(X,r)$ there exists a based simplicial loop $\alpha'$ in $\cRips(X, a + \e)$ such that $\alpha'\simeq \alpha \ \textrm{ rel } \bullet $ in  $\Rips(X,r).$
\end{enumerate}
\end{enumerate}
\end{lemma}

\begin{proof}
(1a) Let $a,b \in X$, $c<d(a,b)<r$, and fix $\nu< r - d(a,b)$. Define set 
$$
A=A(a,b)=\{r'\leq r \ |  \ \exists a',b' \in X \ | \ (a,b) \  \nu\textrm{-descends to } (a',b') \textrm{ and } d(a'b')\leq r' \}.
$$
It is enough to show that $[c,r]\subseteq A$. It is obvious that $A$ is not empty  ($r\in A$) and $A$ is an interval because if $t \in A$ for some $t < r$ then $t' \in A $ for each $t' \in [t,r]$. We proceed by two steps: 
\begin{itemize}
\item[i.] We first prove that if $A=[\rho,r]$ then  $\rho \leq c$.\\
Assume $A=[\rho,r]$ where $\rho > c$. 
There exists a pair $(a',b')$ with $d(a',b')=\rho$, to which $(a,b)$ descends.
Because $\rho$ is not a local minimum we can find $a'',b''\in X$ with $d(a'',b'')< d(a',b') = \rho,\  d(a',a'') < \nu$, and $ d(b',b'')< \nu$. 
Prolonging the mentioned descent by one step using $a''$ and $b''$ we see that $(a,b)$ descends onto $(a'',b''), \ d(a'',b'')\in A$ and thus $A \neq [\rho,r]$.

\item[ii.] We next prove that $A$ is closed at the left endpoint.\\
For  each $n\in \mathbb{N}$ let $(a_n, b_n)$ be a pair in $X$ to which $(a, b)$ descends with $d(a_n, b_n)\leq \rho+ \frac{1}{n}$.
As $X$ is compact the sequences $(a_n)$ and $(b_n)$ have accumulation points $a'$ and $b'$ respectively. Observe that $d(a'b')=\rho$. We claim that $(a, b)$ descends to $(a',b')$. 
Choose $ m \in \mathbb{N}$ such that $d(a_m, a') < \nu$ and $d(b_m, b')<\nu$. We can prolong the descent from $(a, b)$ to $(a_m, b_m)$ by one step to $(a',b')$, which implies $\rho \in A$.
 \end{itemize}
 Following i. and ii. we conclude $[c, r] \subseteq  A$ and thus (1) holds.
 
 (1b) Given a $1$-cycle in $\Rips(X,r)$ replace each its edge $\langle x,y\rangle$ by a finite $c$-sequence $x=x_1, x_2, \ldots, x_p =x', y'=y_p, y_{p-1}, \ldots, y_1=y$ obtained through part (1). The obtained modification preserves the homology class of the $1$-cycle as is evident from Figure \ref{Fig2}. Observe that the sides of the triangles $(x_j, x_{j+1},y_j)$  and $(x_{j+1}, y_j, y_{j+1})$ are at most  $d(x,y)$, $\nu$, and $d(x,y) + \nu$, all of which are smaller than $r$ by the definition of $\nu$. 
 
 (1c) The proof is the same as that of (1b).
 
 The proof of (2) is analogous to that of (1). When $a=0$ we need to choose $\nu < \e$.
 \begin{figure}
    \centering
    \includegraphics{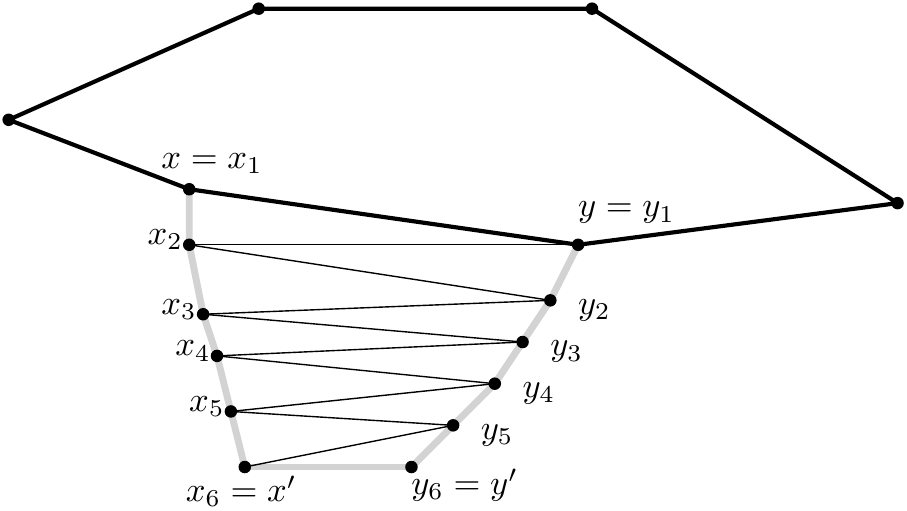}
    \caption{A sketch of statement (1a) of Lemma \ref{LemmaDesc}. Edge $\langle x,y\rangle$ of the bold $1$-cycle is replaced by the grey finite $c$-sequence along the descending finite $\nu$-sequences of part (1a). Such a modification preserves the homology class of a cycle (and based homotopy class of a path) containing the edge $\langle x,y\rangle$.}
    \label{Fig2}
\end{figure}
\end{proof}

\begin{remark}
 Lemma \ref{LemmaDesc} consists of two parts. Part (1) refers to the case when $\locmin(d)\cap [0,r)$ has a maximum (referred to as $c$ in the statement). This is not always the case, see $\{(1/n,\pm (1-1/n))\mid n\in \NN\}\subset \RR^2$. Part (2) considers the case when $\locmin(d)\cap [0,r)$ does not have a maximum (with the supremum being represented by $a$) or when the maximum is zero. 
\end{remark}

\begin{theorem}
\label{ThmDesc} 
Assume $X$ is a compact metric space. Then the emergent $H_1$ spectrum and the emergent $\pi_1$ spectrum are both contained in $\overline{\locmin(d)}$ (when obtained through the open Rips filtration). 
\end{theorem}

\begin{proof}
For $b \notin \overline{\locmin(d)}$ there exists $\e>0$ such that $\locmin(d) \cap (b-\e', b+\e') = \emptyset$. By (2) of Lemma \ref{LemmaDesc} the inclusions $\{i_{t,s} \mid b-\e' < t < s < b+ \e'\}$ induced maps on $H_1$ and $\pi_1$ are surjective. 
\end{proof}

\begin{remark}
When $X$ is not compact Theorem \ref{ThmDesc} might not hold. Observe that if $X \subseteq \RR^n$ is open (and non-empty) then $\locmin(d)=\{0\}$ while persistent homology might be very rich. We also point out that the emergent $H_1$ spectrum might not be contained in emergent $\pi_1$ spectrum as the later only considers loops in the connected component (of the Rips complexes) containing $\bullet$.

When considering persistence obtained through closed Rips filtration ephemeral summands might yield critical values that are not in $\overline\locmin(d)$, see Example \ref{Ex1}. Furthermore, when persistent homology is computed with coefficients from a field, statement (4b) implies that all bars of $H_1$ are closed at the left endpoint if that endpoint is non-zero. If the left endpoint of a bar in $H_1$ is zero, the bar is open at the left endpoint as $\cRips(X,0)$ is a discrete set and thus has trivial $H_1$.
\end{remark}

\subsection{Emergent cycles (loops) and their cardinality}
\label{SubsCard}
In this subsection we prove that one-dimensional homology emerging at a locally isolated local minimum $c$ of $d$ arises by attaching elements of $\M_c$ (see Definition \ref{DefMc}) to $\Rips(X,c)$. We then use this fact to estimate the increase in the rank of $H_1$ at $c$.

\begin{definition}
\label{DefMc}
 For $c\in \locmin(d)$ define 
 $$
\M_c :=\{(x,y) \in X^2 \mid d(x,y)=c, \ d\textrm{ has a local minimum at }(x,y) \}
$$
as the collection of pairs at which $d$ has a local minimum at $d=c$.
\end{definition}

The following are direct consequences of Theorem \ref{ThmDesc} and Lemma \ref{LemmaDesc} .  

\begin{proposition}
\label{PropAdd}
Assume $X$ is a compact metric space. 
Let $c > 0$ be an isolated local minimum of $d$ (i.e., $(c-\e,c+\e) \cap \locmin(d)=\{c\}$) and an emergent critical value of $H_1$.  
 Define
$$
\Rips^*(X,r)=\Rips(X,r) \bigcup_{(x,y)\in \M_c} \langle x,y \rangle.
$$
Then for each $t < s$ within $(c-\e,c+\e)$ the inclusion $\Rips^*(X,t) \hookrightarrow \Rips(X,s)$ induced map on $H_1$ is surjective. 
\end{proposition}

\begin{proposition}
\label{PropAdd1}
Assume $X$ is a compact metric space. 
Let $c > 0$ be an isolated local minimum of $d$ (i.e., $(c-\e,c+\e) \cap \locmin(d)=\{c\}$) and an emergent critical value of $\pi_1$.  
 Define
$$
\Rips^*(X,r)=\Rips(X,r) \bigcup_{(x,y)\in \M_c} \langle x,y \rangle.
$$
Then for each $t < s$ within $(c-\e,c+\e)$ the inclusion $\Rips^*(X,t) \hookrightarrow \Rips(X,s)$ induced map on $\pi_1$ is surjective. 
\end{proposition}

Proposition \ref{PropAdd} states that the emergent homology in persistent $H_1$ at a critical point, which is an isolated local minimum, is generated by complementing the Rips complex by the edges corresponding to the pairs of points at which $d$ attains a local minimum with value $c$. On one hand this yields convenient description of emerging cycles (and loops in persistent fundamental group). On the other hand it allows us to estimate the increase of the rank of $H_1$. 

\begin{definition}
 The \textbf{cardinality of the minimal generating set} of a group $G$ will be denoted by $\mgs(G)$.
\end{definition}

 Clearly $\rank(G) \leq \mgs(G)$ so all subsequent upper bounds on $\mgs$ of groups also hold for the rank.

\begin{proposition}
\label{PropAdd1}
Assume $X$ is a compact metric space. 
Let $c > 0$ be an isolated local minimum of $d$. Then for each $t < s$ within $(c-\e,c+\e)$ we have
$\mgs(H_1(\Rips(X,s))) \leq \mgs (H_1(\Rips(X,t))) + |\M_c|$.
\end{proposition}

\begin{proof}
The proof is a direct consequence of Proposition \ref{PropAdd} as adding  $ |\M_c|$ edges to a simplicial complex increases the cardinality of the minimal generating set of its $H_1$ by at most $ |\M_c|$. 
\end{proof}

\begin{remark}
\label{RemAdd1}
The analogue of Proposition \ref{PropAdd1} for the fundamental group holds only if $\Rips(X,r)$ is connected for some $r<c$. While adding an edge to a simplicial complex increases the cardinality of the minimal generating set of $H_1$ by at most $1$, it might increase the rank of cardinality of the minimal generating set of the fundamental group by more than $1$ if the said edge connects different component of the said simplicial complex. 
\end{remark}

The following corollary refines Proposition \ref{PropAdd1}. It states that $|\M_c|$ is the upper bound for the increase of $\mgs$ of $H_1$ at $c$ plus the decrease of $\rank$ of $H_0$ at $c$.

\begin{corollary}
 \label{CorAdd1}
Assume $X$ is a compact metric space. 
Let $c > 0$ be an isolated local minimum of $d$. Then for each $t < s$ within $(c-\e,c+\e)$ we have
$$
\mgs(H_1(\Rips(X,s))) - \mgs (H_1(\Rips(X,t))) +
$$
$$
+ \rank(H_0(\Rips(X,t))) - \rank (H_0(\Rips(X,s))) \leq |\M_c|.
$$
\end{corollary}

\begin{proof}
The proof is a direct consequence of Proposition \ref{PropAdd} as adding  an edge to a simplicial complex either connects two of its components (thus decreasing the rank of $H_0$ by one) or increases the cardinality of the minimal generating set of its $H_1$ one. 
\end{proof}

\subsection{Reconstruction result for $H_1$ and $\pi_1$}

In this subsection we prove that for a wide class of spaces, the Rips complex at small scales captures $H_1$ and $\pi_1$ of the underlying space $X$. Similar results for fundamental groups of nerve complexes have been proved on numerous occasions, including \cite{Brazas1, CC}. On the other hand, reconstruction results for homotopy type of $X$ using Rips complexes have been proved for certain classes of spaces in \cite{Haus, ZV3}. 

Space $X$ is \textbf{simply connected up to scale $R>0$} if for each $x\in X$ and each positive $r<R$ the open $r$-ball around $x$ is simply connected. In particular this means that such a space is locally path connected. Thus, a compact space which is simply connected up to scale $R>0$ consists of finitely many open path-connected components.

\begin{theorem}
\label{ThmReconstr}
Let $(X, \bullet)$ be a based space which is compact and simply connected up to scale $R>0$. Then $\pi_1(\Rips (X,r),\bullet) \cong \pi_1(X, \bullet)$ for each positive $r < R/3$. Furthermore, if $r' \in (r,R/3)$ then the natural inclusions  $i_{r,r'}$ of Rips complexes at scales $r < r' $ induce isomorphisms on fundamental groups.
\end{theorem}

\begin{proof}
 Define a map $\f=\f_r \colon \pi_1(\Rips(X,r),\bullet)\to \pi_1(X, \bullet)$ by the following rule. If $\alpha$ is a based simplicial loop in $\Rips(X,r)$ given by the sequence of vertices $\bullet = x_1, x_2, \ldots, x_n=\bullet$, define $\f([\alpha])$ as the based homotopy type of the loop $\alpha^X$ obtained as the corresponding concatenation of paths $\psi_k$ in $X$ between points $x_k$ and $x_{k+1}$, where each path $\psi_k$ is contained in $\B(x_k, r)$.
 
We first prove $\f$ is well defined. 
\begin{itemize}
 \item We first show that $\f([\alpha])=[\alpha^X]$ does not depend on the choice of paths $\psi_k$. Assume $\psi'_k$ is a different path from $x_k$ to $x_{k+1}$. Concatenating $\psi_k$ with the reversed path $\psi'_k$ we obtain a loop contained in $\B(x_k, 2r)$. As  $2r < R$ this loop is nullhomotopic. Thus replacing $\psi_k$ by $\psi'_k$ does not change the homotopy type of $\f([\alpha])$ as defined above. 

\item We now show that $\f([\alpha])$ does not depend on the homotopy representative of $[\alpha]$. Assume $\alpha$ and $\beta$ are homotopic based simplicial loops in $\Rips(X,r)$. A homotopy is given by a simplicial map $H$ of a triangulation $\tau$ of $S^1 \times I$ into $\Rips(X,r)$. Define a homotopy $H' \colon S^1 \times I \to X$ on the same triangulation $\tau$ as follows:
\begin{itemize}
 \item For each vertex $v\in \tau$ define $H'(v)=H(v)$. 
 \item Orient all edges in $\tau$. For each oriented edge $\langle a,b \rangle \in \tau$ define $H'$ on $[a,b]$ as a path between $H(a)$ and $H(b)$ in $\B(H(a),r)$.  If $H(a)=H(b)$ choose the constant path. This ensures the obtained homotopy $H'$ is basepoint preserving. 
 \item For each triangle $[a,b,c]\in \tau$ define $H'$ on $[a,b,c]$ as the nullhomotopy in $X$ of the loop $X$ defined by $H'$ on the boundary of $[a,b,c]$. Note that the said loop is contained in the $3r$-ball around one of its vertices (the one from which we can reach the other two vertices along the chosen orientation of the three edges) and the mentioned nullhomotopy exists as $3r < R$.
\end{itemize}
As a result we obtain a based homotopy between $\f([\alpha])$ and $\f([\beta])$.
\end{itemize}
Thus $\f$ is well defined. Furthermore, it is obviously a homomorphism. 

We next prove $\f$ is injective. Let $\alpha$ be a based simplicial loop in $\Rips(X,r)$ given by the sequence of vertices $\bullet = x_1, x_2, \ldots, x_n=\bullet$. Assume $\f([\alpha])$ is contractible. Then there exists a nullhomotopy $H\colon B^2 \to X$ defined on a closed two-dimensional disc $B^2$, whose restriction to the boundary $S^1 = \partial B^2$ is $\alpha^X$. Choose a triangulation $\tau$ of $B^2$ containing vertices $x_1, x_2, \ldots, x_n$ such that for each triangle $T$ in $\tau$ the image of $H(T)$ is contained in a ball of radius $r/2$. Define a simplicial homotopy $H'\colon B^2 \to \Rips(X,r)$, where
\begin{itemize}
 \item the used triangulation on $B^2$ is $\tau$, and
 \item for each vertex $v\in \tau$ define $H'(v)$ to be the vertex in $\Rips(X,r)$ corresponding to the point $H(v)\in X$. 
\end{itemize}
Note that $H'$ is well defined as the vertices of each triangle are a set of diameter less than $r$. Thus $H'$ is a simplicial hullhomotopy of $\alpha'$, which is defined as the restriction of $H'$ to $S^1 = \partial B^2$, in $\Rips(X,r)$. It remains to show that $[\alpha]=[\alpha'] \in \pi_1(\Rips(X,r),\bullet)$. Triangulation $\tau$ restricted to $S^1$ is a refinement of the triangulation on $S^1$ induced by the vertices $x_1, x_2, \ldots, x_n$. Thus the vertices of $\tau$ restricted to $S^1$ are (in a cyclic order so as to determine a simplicial loop $\alpha'$)
$$
\bullet = x_1=x_{1,1}, x_{1,2}, \ldots, x_{1,k_1}= x_2=x_{2,1},  x_{2,2}, \ldots, x_{2,k_2}=x_3=x_{3,1}\ldots, x_n=\bullet.
$$
We claim that for each $j$ the simplicial loop 
$$
x_j=x_{j,1}, x_{j,2}, \ldots, x_{j,k_j}= x_{j+1},x_j
$$
is nullhomotopic in $\Rips(X,r)$. Observe that all the mentioned vertices correspond to points in $\B(x_j,r)$. This means that the mentioned simplicial loop is contained in the closed star of $x_j$ in $\Rips(X,r)$ and is thus contractible. Replacing the portion of the simplicial loop $\alpha'$ between $x_j$ and $x_{j+1}$ by the single edge from $x_j$ to $x_{j+1}$ (as is in $\alpha$) thus preserves the homotopy type. Performing such homotopy-type preserving modification for each $j$ we transform $\alpha'$ into $\alpha$ and conclude the proof of surjectivity.

We now prove $\f$ is surjective. Let $f\colon S^1 \to X$ be a loop based at $\bullet$. For each $t\in S^1$ choose an open interval on $S^1$ containing $t$, such that its image via $f$ is contained in an open $(r/2)$-ball in $X$. By compactness there exists a finite  collection of such intervals covering $S^1$. Denoting the centers of the obtained intervals by $t_j$, we obtain a finite sequence of  points $\bullet = t_1, t_2, \ldots, t_k = \bullet$ appearing in the positive order on $S^1$ and winding around it exactly once, such that for each $k$ the image via $f$ of the closed interval from $t_k$ to $t_{k+1}$ (in the positive direction) is contained in $\B(t_k,r)$. The finite strictly $r$-cycle $\bullet = f(t_1), f(t_2), \ldots, f(t_k) = \bullet$ is a simplicial based loop in $\Rips(X,r)$ whose based homotopy class is mapped to $f$ via $\f$.

Thus $\pi_1(\Rips (X,r),\bullet) \cong \pi_1(X, \bullet)$ for all positive $r < R/3$. As maps $\f_r$ commute with the inclusions induced maps on the fundamental groups of Rips complexes, the second conclusion of the theorem also holds. 
\end{proof}

\begin{theorem}
\label{ThmReconstr1}
Let $X$ be a based space which is compact and simply connected up to scale $R>0$. Then $H_1(\Rips (X,r)) \cong H_1(X)$ for each positive $r < R/3$. Furthermore, if $r' \in (r,R/3)$ then the natural inclusions  $i_{r,r'}$ of Rips complexes at scales $r < r' $ induce isomorphisms on $H_1$.
\end{theorem}

\begin{proof}
 As $X$ is compact and locally path connected it consists of a finite number of compact path connected components $A_1, A_2, \ldots, A_n$. It follows from definition of $R$ that the distinct components are at distance at least $R$ and thus $\Rips(X,r)$ is a disjoint union of subcomplexes $\Rips(A_j,r)$. Without loss of generality we may assume $\bullet \in A_1$. Theorem  \ref{ThmReconstr1} coupled with Hurewicz theorem states that $H_1(\Rips (A_1,r)) \cong H_1(A_1)$. Changing basepoint to a different component $A_j$ we obtain $H_1(\Rips (A_j,r)) \cong H_1(A_j)$. As $j$ was arbitraty we conclude the isomorphism $H_1(\Rips (X,r)) \cong H_1(X)$. 
 
 The second part follows similarly from the analogous property of the fundamental group.
\end{proof}

\begin{remark}
 Theorems \ref{ThmReconstr} and  \ref{ThmReconstr1} state that initially the persistent $\pi_1$ and $H_1$ are constant and isomorphic to the corresponding invariant of $X$. 
\end{remark}

\subsection{Bounds on the generating sets of $1$-dimensional persistence}
In this subsection we combine the reconstruction results of the previous subsection with the behaviour of $1$-dimensional persistence at critical scales as discussed in Subsection \ref{SubsCard} to provide a global bound on the rank of $H_1$ the fundamental group of Rips complexes.

Given a space $K$ with components $K_1, K_2, \ldots, K_m$ and $x_j\in K_j, \forall j$,  define 
$$
\mathrm{MGS(K)} = \sum_{j=1}^m \mgs(\pi_1(K_j,x_j)).
$$

\begin{remark}
 \label{RemMGS}
 Adding an edge $E$ to a simplicial complex $K$ can increase the fundamental group in two ways (see also Remark \ref{RemAdd1}):
\begin{enumerate}
 \item If both endpoints of $E$ are in the same component $B$, then the addition of $E$ increases the rank of the fundamental group of $B$ (with a basepoint in $B$) by one.
 \item If the endpoints of $E$ are in different components $B_j \neq B_k$, then the addition of $E$ connects the two components and increases the rank of the fundamental group as follows:
 	\begin{enumerate}
 	\item For a basepoint $b\notin B_j \cup B_k$, the rank of $\pi_1(K,b)$ does not change.
	\item For basepoints $b_j\in B_j$ and $b_k \in B_k$,  $\mgs(\pi_1(K, b_j))$ increases by  $\mgs(\pi_1(B_k, b_k))$.
	\end{enumerate}
\end{enumerate}
In particular, adding an edge to a simplicial complex $K$ increases $\mathrm{MGS(K)}$ by at most one. 
\end{remark}

\begin{remark}
\label{RemFinite}
If $Y$ is a compact, connected, locally path connected metric space, then $\mgs(\pi_1(Y))$ is known to be finite, see for example \cite{Dyd}. Consequently, each space $X$, which is compact and simply connected up to scale $R>0$, has finite $\mgs(\pi_1(X))$, $\mathrm{MGS(X)}$ and $\mgs(H_1(X))$. In particular, such a space consists of finitely many compact components, which are simply connected up to scale $R>0$. Applying the first statement of this remark to each of the components (and the Hurewicz's theorem for the homological version) we obtain finiteness of all mentioned invariants.
\end{remark}

\begin{theorem}
\label{ThmFG}
Let $(X, \bullet)$ be a based space which is compact and simply connected up to scale $R>0$. Assume $\locmin(d)$ is finite. If $M_c$ is finite for each $c > 0$, then for each $r >0$:
 $$
 \mgs(H_1(\Rips(X,r))) \leq \mgs (H_1(X)) + \sum_{c<r} |\M_c| < \infty.
 $$
 Furthermore, if $A_1, A_2, \ldots, A_n$ are the path connected components and $x_j\in A_j$, then for each $r>0$:
  $$
 \mathrm{MGS}(\Rips(X,r)) \leq \mathrm{MGS(X)} + \sum_{c<r} |\M_c|< \infty.
 $$
 and in particular, 
 $$
 \mgs(\pi_1(\Rips(X,r),\bullet)) \leq \mathrm{MGS(X)} + \sum_{c<r} |\M_c|< \infty.
 $$
\end{theorem}

\begin{proof}
 By Theorem \ref{ThmDesc}  $ \mgs(H_1(\Rips(X,r)))$ may increase only at $\locmin(d)$. For small positive $r$ we have $ \mgs(H_1(\Rips(X,r))) = \mgs (H_1(X))$ by Theorem \ref{ThmReconstr1}. The increase of $ \mgs(H_1(\Rips(X,r)))$ at each point of $\locmin(d)$ is bounded from above by $|\M_c|$ by Proposition \ref{PropAdd1}, which proves the first part of the theorem.
 
By Remark \ref{RemMGS}, Theorem \ref{ThmDesc}, and Theorem \ref{th:1},   $  \mathrm{MGS}(\Rips(X,r))$ may increase only at $\locmin(d)$. Without loss of generality we have $\bullet = x_1$.  For small positive $r$, $\Rips(X,r)$ is the disjoint union of $\Rips(A_1,r), \Rips(A_s,r), \ldots, \Rips(A_n,r)$ and thus  $\pi_1(\Rips(X,r),\bullet) = \pi_1(\Rips(A_1,r),\bullet)$ by Theorem \ref{ThmReconstr}. Hence  $\mgs(\pi_1(\Rips(X,r),\bullet)) = \mgs (\pi_1(A_1,\bullet))$ for small positive $r$ and by extension of the argument, $\mathrm{MGS}(\Rips(X,r)) = \mathrm{MGS} (X)$ for small positive $r$. The increase of $ \mathrm{MGS}(\Rips(X,r))$ at each point of $\locmin(d)$ is generated by adding $|\M_c|$ edges by Proposition \ref{PropAdd1} and thus by Remark \ref{RemMGS}, $ \mathrm{MGS}(\Rips(X,r))$ increases at $c$ by at most $|\M_c|$. 

The last statement holds as $ \mgs(\pi_1(\Rips(X,r),\bullet)) \leq  \mathrm{MGS}(\Rips(X,r)).$ The finiteness of the mentioned invariants follows from Remark \ref{RemFinite}.
\end{proof}

\begin{corollary}
\label{CorFinRankFinal}
 Let $(X, \bullet)$ be a based space which is compact and simply connected up to scale $R>0$, $|\locmin(d)| < \infty$, and  $|M_c| < \infty, \forall c > 0$. Then for each $r$:
 $$
 \rank(H_1(\Rips(X,r))) < \infty \quad\textrm{ and } \quad \rank(\pi_1(\Rips(X,r)), \bullet)<\infty.
 $$
\end{corollary}

\begin{proof}
 The statement follows from Theorem \ref{ThmFG}.
\end{proof}

Example \ref{Ex3} in the last section demonstrates that $ \rank(H_1(\Rips(X,r)))$ may be infinite if $X$ is a compact metric space.  

\subsection{Combinatorial criterion for spectrum}

Up to this point our results explore and exploit the fact that $H_0$ spectrum and emergent one-dimensional spectrum of a compact metric space are in a sense induced by  members of $\M_c$. It seems to be more complicated to provide a sufficient condition that will imply a particular element of $\M_c$ induces any mentioned critical effect. In this subsection we provide such a condition in case $|\M_c|=1$ for an isolated local minimum $c$. Observe that this case includes every finite $X \subset \RR^n$ in general position.

\begin{theorem}
\label{ThmConverse}
Assume $X$ is a compact metric space. Let $c > 0$ be an isolated local minimum of $d$ and $\M_c= \{(x,y)\}.$ Then the following are equivalent:
\begin{description}
 \item[(a)] $c$ is a member of either $H_0$ spectrum or emergent $H_1$ spectrum of $X$.
 \item[(b)] There does not exist $z\in X \setminus \{x,y\}$ such that $d(z,x) \leq c$ and $d(y,z) \leq c$.
\end{description}
\end{theorem}

\begin{proof}
\textbf{(a)} $\implies$ \textbf{(b)}: Clearly $c$ cannot be in both $H_0$ spectrum and emergent $H_1$ spectrum of $X$ as $\M_c= \{(x,y)\}$ as adding an edge either connects two components or increases the $H_1$ (see Proposition \ref{PropAdd}).

 Let $c$ be a member of $H_0$ spectrum. Then $x \in K_x$ and $y \in K_y$, where $K_x$ and $K_y$ are different connected components of $\Rips(X,r)$ at scale $r$ just before $c$. By Lemma \ref{le:1a} the vertex sets of $K_x$ and $K_y$ in $X$, denoted by $X_x \subset X$ and $X_y \subset Y$ respectively, are compact open subsets in $X$ at distance $c$. Assume there exists $z\in X \setminus \{x,y\}$ such that $d(z,x) \leq c$ and $d(y,z) \leq c$. 
\begin{itemize}
 \item If $z\in X_x$ then $d(y,z) = c$ as the distance between $X_x$ and $X_y$ is $c$. But then $(y,z)$ is a local minimum of $d$ different from $(x,y)$, a contradiction with $|\M_c|=1$.
 \item Case $z\in X_y$ is treated similarly.
 \item Let $z$ be in a different $\sim_c$ equivalent class $X_z\subset X$ than $X_x$ and $X_y$. Then $X_z$ is at distance at least $c$ from $X_x$ and $X_y$ due to definition of $\sim_c$, hence $d(x,z)=c$. Similarly as in the first item above, this implies $(x,z)$ is a local minimum of $d$, a contradiction. 
\end{itemize}
Hence the assumed $z$ may not exist.

Let $c$ be a member emergent $H_1$ spectrum and let $[\alpha]$ be a homology class of a connected cycle emerging at $c$.  By Lemma \ref{LemmaDesc}, $\alpha$ can be chosen as a finite $c$-cycle in $X$ containing a non-zero multiple of oriented edge $ \langle x,y \rangle$. Assume there exists $z\in X \setminus \{x,y\}$ such that $d(z,x) \leq c$ and $d(y,z) \leq c$. Then each occurrence of $ \langle x,y \rangle$ within $\alpha$ can be replaced using $\partial \langle x, z, y\rangle$, in effect replacing edge $\langle x,y \rangle$ with the edges $\langle x,z \rangle$ and $\langle z,y \rangle$, without changing $[\alpha]$. We can now use the procedure of Lemma \ref{LemmaDesc} to replace each edge within $\alpha$ at distance $c$ by a finite strictly $c$-sequence without changing $[\alpha]$ as such pairs are not local minima of $d$. Thus we have constructed a representative of $[\alpha]$ appearing at scale smaller than $c$, hence $[\alpha]$ has not emerged at $c$. As a result, the assumed $z$ can not exist.
 
\textbf{(b)} $\implies$ \textbf{(a)}:
Assume \textbf{(a)} does not hold. 
\begin{enumerate}
 \item As $c$ is not in $H_0$ spectrum, the pair $\{x,y\}$ does not connect two different components in $\Rips(X,c)$. Thus there exists a finite strictly $c$-sequence $x=x_0, x_1, \ldots, x_m=y$. Let $\alpha$ be the simplicial cycle corresponding to the finite $c$-loop $x=x_0, x_1, \ldots, x_m=y, x$.
 \item  By our assumption $[\alpha]$ is nullhomologous in $\Rips(X,r)$ for each $r>c$. 
 
 \item As $c \in \locmin(d)$ and $|\M_c|=1$ we can choose $R < c/4$ such that $d(x',y')> c,\  \forall x' \in \cB(x,R)\setminus \{x\}, \forall y' \in \cB(y,R)\setminus \{y\}$.
 
 \item Define compact sets $A_x= \cB(x,R) \setminus \B(x,R/2)$ and $A_y= \cB(y,R) \setminus \B(y,R/2)$. Note that by compactness and $|\M_c|=1$, there exist real numbers $w_1<w_2$ such that $d(A_x \times \cB(y,R) \cup \cB(x,R) \times A_y) \subseteq [w_1, w_2] \subset (c,\infty).$
 
 \item Choose $\e_0 > 0$ so that for each $x' \in \cB(x,R)$ and $y' \in \cB(y,R)$ with $d(x', y') \in (c, c + \e_0)$ we have $x' \in \B(x,R/2)$ and $y' \in \B(y,R/2)$. If such $\e_0$ didn't exist we would have for each sufficiently large $n \in \NN$, a pair $t_n\in \cB(x,R)$ and $s_n \in \cB(y,R)$ with $d(t_n, s_n) \in (c, c + 1/n)$ and either $t_n \in A_x$ or $s_n \in A_y$. Without loss of generality we could choose a subsequence $(n_j)_{j\in \NN}$ of $\NN$ so that $t_{n_j} \in A_x, \forall j$. Then we have $t=\lim_{j \to \infty} t_{n_j}\in A_x$, $s=\lim_{j \to \infty} s_{n_j}\in \cB(y,R)$, and $d(t,s)=c$,  a contradiction with $|\M_c|=1$.
 
 \item Let $r\in (c,c + \e_0)$. For each $1$-chain $\beta = \sum_{j=1}^k \mu_j, \sigma_j$ in $\Rips(X,r)$ (i.e., $\mu_j \in G$ and $\sigma_j$ is an oriented $1$-simplex in $\Rips(X,r)$) we define the invariant $N_r(\beta)$ as follows. Define $F_r$ as the collection of those indices $j \in \{1,2,\ldots, k\}$ for which the first vertex of  $\sigma_j$ is in $\B(x,R/2)$ and the second is in $\B(y,R/2)$. Similarly, let $L_r$ be the collection of those indices $j \in \{1,2,\ldots, k\}$ for which the second vertex of  $\sigma_j$ is in $\B(x,R/2)$ and the first is in $\B(y,R/2)$. Define 
 $$
 N_r(\beta) = \sum_{j\in F_r} \mu_j - \sum_{j \in L_r} \mu_j \in G.
 $$
 Quantity $ N_r(\beta)$ represents the total amount of weights in $G$ pointing from $\B(x,R/2)$ to $\B(y,R/2)$ along oriented edges of $\beta$.
 
 \item Fix $r\in (c,c + \e_0)$. Observe that $N_r(\alpha)=1$. On the other hand, (2) implies that $\alpha$ is a boundary of a $2$-cycle in $\Rips(X,r)$ and hence the $N_r$ of the said boundary is also $1$. Thus there exists at least one $2$-simplex with non-trivial $N_r$. A simple case analysis shows that this is possible if and only if the $2$-simplex in question has:
\begin{itemize}
 \item One vertex, say $a_r$, in $\B(x,R/2)$.
  \item One vertex, say $b_r$, in $\B(y,R/2)$.
   \item One vertex, say $z_r$, in $X \setminus \Big(\B(x,R/2) \cup \B(y,R/2) \Big )$.
\end{itemize}
As $ \{a_r, b_r, z_r \}$ forms a simplex in $\Rips(X,r)$, the pairwise distances are smaller than $r$. On the other hand, (5) implies $z_r \notin \B(x,R) \cup \B(y,R)$ thus $d(a_r, z_r) \geq R/2$ and $d(b_r, z_r) \geq R/2$.

\item By compactnes of $\B(x,R/2), \B(y,R/2)$, and $X \setminus (\B(x,R) \cup \B(y,R))$, there exists a subsequence $(\ell_j)_{j\in\NN}$ of $\NN$ such that the following limits exist in the corresponding mentioned compact sets:
$$
a=\lim_{j \to \infty} a_{c + 1/\ell_j} \quad b=\lim_{j \to \infty} b_{c + 1/\ell_j}, \quad z=\lim_{j \to \infty} a_{c + 1/\ell_j}.
$$

\item By (3) we have $a=x$ and $b=y$. By the upper bounds on pairwise distances in (7) we have $d(a,z) \leq r $ and  $d(b,z) \leq r $.  Furthermore, the lower bounds in (7) imply $d(a,z) \geq R/2$ and $d(b,z) \geq R/2$, which in particular mean $z\in X \setminus \{x,y\}$. Thus \textbf{(b)} does not hold.
\end{enumerate}
 \end{proof}

If $X$ is connected, then replacing cycles by simplicial loops in the above proof yields the analogous result for emergent $\pi_1$ spectrum. 

\begin{proposition}
\label{PropConversePi}
Assume $X$ is a connected compact metric space. Let $c > 0$ be an isolated local minimum of $d$ and $\M_c= \{(x,y)\}.$ Then $c$ is a member of $H_0$ spectrum or emergent $\pi_1$ spectrum of $X$ iff there exists no $z\in X \setminus \{x,y\}$ such that $d(z,x) \leq c$ and $d(y,z) \leq c$.
\end{proposition}

\section{Detecting all local minima of the distance function}
\label{SectsRips}

Theorem \ref{ThmConverse} provides a condition under which certain locally isolated elements of $\locmin(d)$ are detected via persistence. In this section we prove that persistence can in fact detect each member of $\locmin(d)$ of finite $\M_c$, if we use appropriate selective Rips complexes instead of Rips complexes. Selective Rips complexes have been introduced in \cite{ZV4} and represent subcomplexes of Rips complexes with controllably thin simplices (see also \cite{Lem} for a corresponding reconstruction result). The motivation for their construction was to provide a flexible construction of filtrations closely related to Rips filtrations, which  enables us to detect as many geodesic circles (i.e., isometric images of circles equipped with a geodesic metric, inside a geodesic space) as possible using persistence. In this section we use a similar approach to detect $\locmin(d)$.

\begin{definition} \cite{ZV4}
 \label{DefSRips}
Let $Y$ be a metric space, $r_1 \geq r_2, n\in \NN$.  \textbf{Selective Rips complex} $\sRips(Y; r_1, n, r_2)$ is an abstract simplicial complex defined by the following rule: a finite subset $\sigma\subset Y$ is a simplex iff the following two conditions hold:
\begin{enumerate}
 \item $\diam (\sigma) < r_1$;
 \item there exist subsets $U_0, U_1, \ldots,  U_n\subset U$ of diameter less than $r_2$ such that $\sigma \subset U_0 \cup U_1 \cup \ldots \cup U_n$.
\end{enumerate} 
\end{definition} 

The geometric intuition behind Definition \ref{DefSRips} is that  simplices of dimension above $n$ in $\sRips(Y; r_1, n, r_2)$ are very thin, and up to ``distortion'' $r_2$ close to an $n$-dimensional simplex in $\Rips(X,r_1)$. 
Observe that $\sRips(Y; r_1, n, r_2) \leq \Rips(Y,r).$ In this paper we will be using selective Rips complexes of form $\sRips(Y; r_1, 1, r_2)$, which means that $2$-simplices will be thin, i.e., that the shortest side of a $2$-simplex will be smaller than $r_2$. In order to simplify the notation of filtrations by selective Rips complexes we will focus on filtrations of a form 
$$
\mathcal{F}=\{\sRips(Y;r,1,r_2(r))\}_{r \geq 0},
$$
Where $r_2=r_2(r) \colon [0,\infty) \to [0,\infty)$ is a strictly increasing continuous bijection satisfying $r_2(r) \leq r$.

Many of the results of the previous sections also hold for selective Rips complexes:
\begin{itemize}
 \item The $H_0$ persistence of the Rips filtration of $X$ is isomorphic to the $H_0$ persistence of the selective Rips filtration $\mathcal{F}$ as the one dimensional skeletons of $\Rips(X,r)$ and $\sRips(Y;r,1,r_2(r))$ coincide. Thus Theorem \ref{th:1} also holds for selective Rips complexes arising from filtration $\mathcal{F}$.
 \item Lemma \ref{LemmaDesc} also holds for selective Rips complexes arising from filtration $\mathcal{F}$. The reason is that the descending condition of Definition \ref{DefDescent} is established for arbitrarily small positive $\nu$ and is proved as such in 1(a) of Lemma \ref{LemmaDesc}. Consequently, the $2$-simplices used to prove parts 1(b) and 1(c) of Lemma \ref{LemmaDesc} (see Figure \ref{Fig2}) can be taken to be as thin as required by $\mathcal{F}$.
\end{itemize}

As a result we obtain the following generalization of Theorem \ref{ThmDesc}.

\begin{theorem}
\label{ThmDescSRips} 
Assume $X$ is a compact metric space. Then the emergent $H_1$ spectrum and the emergent $\pi_1$ spectrum arising from filtration $\mathcal{F}$ are both contained in $\overline{\locmin(d)}$. 
\end{theorem}

We are now in a position to prove the main result of this section: each member $c \in \locmin(d)$ can be detected by persistence via selective Rips complexes if $\M_c$ is finite. In case $\M_c$ is infinite Example \ref{Ex2} shows that $c$ may not be detected even by selective Rips complexes.

\begin{theorem}
\label{ThmSRips1} 
Assume $X$ is a compact metric space, $c \in \locmin(d)$, and $\M_c$ is finite. There exists a filtration $\mathcal{F}$ with function $r_2$, such that $c$ is either a critical value of $H_0$ or an emergent critical value of $H_1$ arising from filtration $\mathcal{F}$.
\end{theorem}

The following proof shares some of the setup with the proof of \textbf{(b)} of Theorem \ref{ThmConverse}. The general idea is inspired by the proof of the main result of \cite{ZV4}.

\begin{proof}
Choose $(x,y)\in \M_c$. Without loss of generality we can assume $c$ is not in $H_0$ spectrum. Thus $c>0$ and there exists a finite strictly $c$-sequence $x=x_0, x_1, \ldots, x_m=y$ in $X$. Let $\alpha$ be the simplicial cycle corresponding to the finite $c$-loop $x=x_0, x_1, \ldots, x_m=y, x$. We claim that there exists function $r_2$ such that $[\alpha]\in H_1\big(\sRips(Y; r,1,r_2(r))\big)$ is nontrivial for $r\in (c,c+\e)$ for some $\e>0$.
\begin{enumerate}
 \item Choose $R < c/4$ such that $d(x',y')> c,\  \forall x' \in \cB(x,R)\setminus \{x\}, \forall y' \in \cB(y,R)\setminus \{y\}$.
 
 \item Define compact sets $A_x= \cB(x,R) \setminus \B(x,R/2)$ and $A_y= \cB(y,R) \setminus \B(y,R/2)$. Note that by compactness and finiteness of $\M_c$, there exist real numbers $w_1<w_2$ such that $d(A_x \times \cB(y,R) \cup \cB(x,R) \times A_y) \subseteq [w_1, w_2] \subset (c,\infty).$
 
 \item Choose $\e > 0$ so that for each $x' \in \cB(x,R)$ and $y' \in \cB(y,R)$ with $d(x', y') \in (c, c + \e)$ we have $x' \in \B(x,R/2)$ and $y' \in \B(y,R/2)$.
 
 \item Choose $r_2$ so that $r_2(c+ \e) < R/2$. This is the only additional condition that will be imposed on $r_2$. 

\item Let $r\in (c,c + \e)$. For each $1$-chain $\beta = \sum_{j=1}^k \mu_j, \sigma_j$ in $\sRips(X;r,1,r_2(r))$ (i.e., $\mu_j \in G$ and $\sigma_j$ is an oriented $1$-simplex in $\sRips(X;r,1,r_2(r))$) we define the invariant $N_r(\beta)$ as follows. Define $F_r$ as the collection of those indices $j \in \{1,2,\ldots, k\}$ for which the first vertex of  $\sigma_j$ is in $\B(x,R/2)$ and the second is in $\B(y,R/2)$. Similarly, let $L_r$ be the collection of those indices $j \in \{1,2,\ldots, k\}$ for which the second vertex of  $\sigma_j$ is in $\B(x,R/2)$ and the first is in $\B(y,R/2)$. Define 
 $$
 N_r(\beta) = \sum_{j\in F_r} \mu_j - \sum_{j \in L_r} \mu_j \in G.
 $$
 
 \item Fix a scale $r\in (c,c + \e)$. We now demonstrate that the boundary of a $2$-cycle in $\sRips(X;r,1,r_2(r))$ has even $N_r$. By additivity it suffices to consider a single $2$-simplex. A simple case analysis shows that a $2$-simplex in $\sRips(X;r,1,r_2(r))$ has an odd $N_r$ if:
\begin{itemize}
 \item One vertex, say $a_r$, in $\B(x,R/2)$.
  \item One vertex, say $b_r$, in $\B(y,R/2)$.
   \item One vertex, say $z_r$, in $X \setminus \Big(\B(x,R/2) \cup \B(y,R/2) \Big )$. By (2) this implies $z_r \in X \setminus \Big(\B(x,R) \cup \B(y,R) \Big )$.
\end{itemize}
This would mean all three pairwise distances in such a simplex would be larger that $R/2$, a contradiction with (4). Hence such $2$-simplex does not exist.

\item Observe that $N_r(\alpha)=1$ and thus $\alpha$ can not be expressed as a boundary of a $2$-chain in $\sRips(X;r,1,r_2(r))$ by (6). This proves our claim that $[\alpha]\in H_1\big(\sRips(Y; r,1,r_2(r))\big)$ is nontrivial for $r\in (c,c+\e)$. As for each $r'<c$ every $1$-chain in  $\sRips(Y; r',1,r_2(r'))$ has trivial $N_r$, we also conclude $[\alpha]$ emerges at $c$. This concludes the proof.
\end{enumerate}
\end{proof}

The following theorem states that if the distance function $d$ attains a positive local minimum in at most finitely many pairs of points in $X$, then all local minima of $d$ can be detected using persistence with a single selective Rips filtration $\mathcal{F}$.

\begin{theorem}
\label{ThmSRips2} 
Assume $X$ is a compact metric space, $\locmin(d)$ is finite, and for each positive $c \in \locmin(d)$ the set $\M_c$ is finite. Then there exists a filtration $\mathcal{F}$ with function $r_2$, such that each element of $\locmin(d)$ is either a critical value of $H_0$ or an emergent critical value of $H_1$ arising from filtration $\mathcal{F}$. In particular, $\locmin(d)$ is the union of $H_0$ spectrum and emergent $H_1$ spectrum of $X$ via filtration $\mathcal{F}$.
\end{theorem}

\begin{proof}
Assume scale  $c\in \locmin(d)$ is not in $H_0$ spectrum. By Theorem \ref{ThmSRips1} scale  $c$ is in emergent $H_1$ spectrum via $\mathcal{F}$ if condition (4) of the proof of Theorem \ref{ThmSRips1} holds. As it is easy to satisfy finitely many such conditions simultaneously, there exists $r_2$ such that the  theorem holds. 
\end{proof}

\section{(Counter)examples}
\label{SectEx}

In this section we present three examples that demonstrate the necessity of some of the assumptions in our results. 

\begin{example}
 \label{Ex1}
\textbf{Closed Rips filtrations may induced critical values not in $\overline{\locmin(d)}$.} The left part of Figure \ref{FigEx1} shows space $A$ as a solid curve. The dashed circular arcs are parts of circles with centers at points $a$ and $b$. Note that the pair $(a,b)$ is not a local minimum of $d$, neither is it in $\overline{\locmin(d)}$. However, the closed Rips filtration of $A$ still has $d(a, b)$ as an emergent $H_1$ value as $\cRips(A, d(a,b))$ contains the edge $\langle a,b \rangle$ as a maximal simplex. This demonstrates that Theorem \ref{ThmDesc} does not hold for closed Rips filtrations. This example appeared first in \cite{ZV2}.

The non-trivial class $H_1(\cRips(A,d(a,b)))$ emerging at $d(a, b)$ is trivial in $H_1(\cRips(A,r))$ for all $r> d(a, b)$, which means its lifespan is zero. As open and closed Rips filtrations are $0$-interleaved, the zero-lifespan elements (also referred to as ephemeral summands) are the only way in which the spectrum of a closed Rips filtration can be larger than the spectrum of the open Rips filtration, see \cite{ChaObs, Cha2} for details and definitions of the mentioned terms.
\end{example}

\begin{example}
 \label{Ex2}
 \textbf{A local minimum $c$ of $d$ may not be detectable if $\M_c$ is infinite.} The right part of Figure \ref{FigEx1} shows space $B$. It consists of three line segments, two of which are parallel. The distance function $d$ attains a local minimum at pair $(a,b)$, yet for each selective Rips filtration of $B$ the scale $d(a,b)$ is not in spectrum of $B$. This demonstrates that the requirement $|\M_c| < \infty$ in Theorem \ref{ThmSRips2}  is necessary.
\end{example}

 \begin{figure}
     \label{FigEx1}
    \centering
    \includegraphics{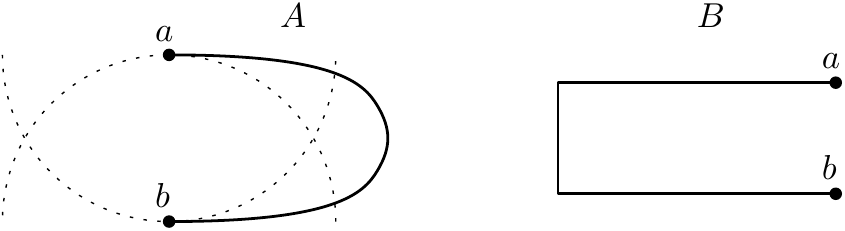}
    \caption{ Spaces of Examples \ref{Ex1} (left) and \ref{Ex2} (right), both are subspaces of the Euclidean plane. }
\end{figure}

\begin{example}
 \label{Ex3}
 \textbf{$H_1(\Rips(X,r))$ may have infinite rank even if $X$ is compact.} Example $Z = U \sqcup L$ is sketched on the left of Figure \ref{FigEx3}. As a set it consists of lower points $L$ and upper points $U$ in the plane, defined as
 $$
 L=\big\{x_j \mid j\in \{\infty, 1,2, \ldots\}\big\}, \quad U= \big\{y_j \mid j\in \{\infty, 1,2, \ldots\}\big\}.
 $$ 
 The metric used is the Manhattan metric denoted by $d=d_1$, i.e., 
 $$
 d((a_1, a_2), (b_1, b_2))= |a_1 - b_1| +|a_2 - b_2|.
 $$ 
 The details on local distances are indicated on the right side of Figure \ref{FigEx3}. Note that $\lim_{j \to \infty} x_j = x_\infty$ and $\lim_{j \to \infty} y_j = y_\infty$. As $\sum_{j=1}^\infty 2^{-j}=1$ we see that $d(x_1, x_\infty)=d(y_1, y_\infty)=5/4$. Also note that:
\begin{enumerate}
\item $d(x_j, x_k) \leq 5/4, \forall j,k \in \{\infty, 1,2, \ldots\}$.
 \item  $d(x_1, y_1)=r-1$ and also $d(x_j,y_j)<r, \ \forall j<\infty$.
 \item $d(x_j, y_{j+1})=d(y_j, x_{j+1})=r + 2^{-(n+2)}>r, \ \forall j<\infty$.
 \item Observations (2) and (3) imply $d(x_j, y_k) < r$ iff $j=k$.
\end{enumerate}
Fixing $r \geq 5/4$ we make the following observations for $\Rips(Z,r)$:
\begin{itemize}
 \item $\Rips(Z,r)$ contains the full simplex on $L$ as $\diam(L) = 5/4$ by (1). On a similar note, $\Rips(Z,r)$ contains the full simplex on $U$.
 \item Observations (2), (3), and (4) imply that the only simplices connecting $U$ and $L$ in $\Rips(Z,r)$ are the ``vertical'' edges $\langle x_j, y_j \rangle$ for all $j < \infty$.
\end{itemize}
As a result $H_1(\Rips(Z,r),\mathbb{Z}) \cong \bigoplus_{j=1, 2, \ldots} \mathbb{Z}$ is the countable direct sum of integer groups, which is not finitely generated. 
 
 This example complements analogue examples on closed Rips complexes in \cite{Cha2}. It also shows that finiteness of $\locmin(d)$ is required in Corollary \ref{CorFinRankFinal}. On a different note, observe that $r=5/4$ is an emergent critical scale of $H_1$ despite not being a local minimum of $d$. 
\end{example}

 \begin{figure}
    \centering
    \includegraphics{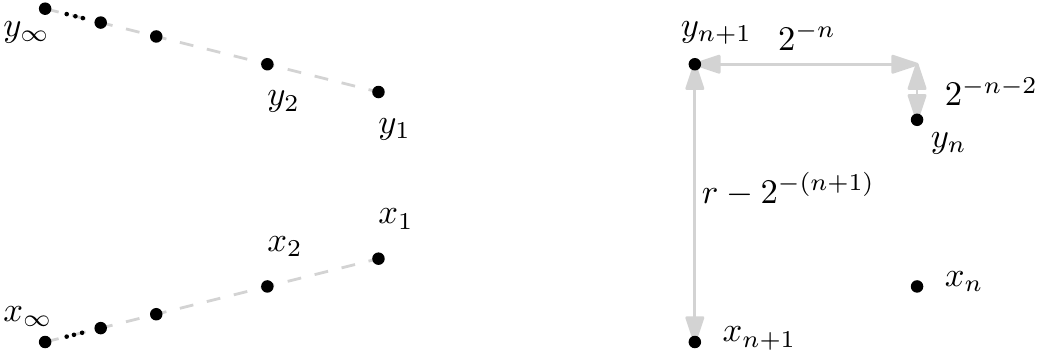}
    \caption{Space $Z$ of Example \ref{Ex3} on the left, and details of tis local configuration on the right. The metric used is $d_1$.}
    \label{FigEx3}
\end{figure}

\begin{example}
The results of this paper imply that a local minimum of $d$ can  ``contribute'' two potential changes to persistence (as a critical edge): an increase in $H_1$ or a decrease in $H_0$. However, it turns out that when $\M_c$ is infinite, it can actually affect persistent homology in any way, i.e., terminating or giving rise to homology in any dimension. Hence a general theory of critical edges is much more complicated and not at all analogous to persistence on finite metric spaces.  
 
 For example, let $A$ be a planar circle of radius $10$ and $x\in A$. Define 
 $$
 B = A \times \{0,1\} \cup \{x\} \times [0,1] \subset \RR^3
 $$
 as two parallel circles connected by a line segment.
 Observe that $B$ is a connected space with the first Betti number $2$. By Theorem \ref{ThmReconstr},  $H_1(\Rips(B,r))$ is also of rank $2$ for $r < 1$. However, for $\in (1,2)$ the local minima attained at the uncountably many pairs $\{y\}\times \{0,1\}$ for $y\in A \setminus \{x\}$ ``stack up'' in circle to terminate a one-dimensional homology class and result in a decrease in the rank $H_1(\Rips(B,r))$ to $1$ at $r=1$. 
\end{example}


\end{document}